\documentclass{acm_proc_article-sp}

\usepackage{amsfonts}
\usepackage{amssymb}
\usepackage{bbm}
\usepackage{etoolbox}

\newcommand{\R}{\mathbb{R}}

\newcommand{\Z}{\mathbb{Z}}

\newcommand{\E}{\mathbb{E}}

\newcommand{\T}{\mathbb{T}}

\renewcommand{\S}{\mathbb{S}}

\renewcommand{\P}{\mathbb{P}}

\newcommand{\calH}{\mathcal{H}}

\newcommand{\calA}{\mathcal{A}}
\newcommand{\calC}{\mathcal{C}}

\newcommand{\normal}{\mathcal{N}}

\newcommand{\pdiff}[2]{\frac{\partial #1}{\partial #2}}
\newcommand{\pdiffII}[3]{\ifstrequal{#2}{#3}
{\frac{\partial^2 #1}{\partial #2^2}}
{\frac{\partial^2 #1}{\partial #2 \partial #3}}
}
\newcommand{\diffII}[3]{\ifthenelse{\equal{#2}{#3}}
{\frac{d^2 #1}{d #2^2}}
{\frac{d^2 #1}{d #2 d #3}}
}
\newcommand{\diff}[2]{\frac{d #1}{d #2}}

\newcommand{\grad}{\nabla}

\renewcommand{\Pr}{\text{Pr}}

\newcommand{\symdiff}{\Delta}

\newtheorem{theorem}{Theorem}[section]

\newtheorem{lemma}[theorem]{Lemma}

\newtheorem{remark}[theorem]{Remark}

\usepackage[ruled]{algorithm2e}
\DeclareMathOperator{\NS}{NS}

\title{Testing surface area with arbitrary accuracy}
\numberofauthors{1}
\author{Joe Neeman \\
       \affaddr{Department of Electrical and Computer Engineering, and}\\
       \affaddr{Department of Math}\\
       \affaddr{University of Texas at Austin}\\
       \email{joeneeman@gmail.com}
}

\begin{document}

\maketitle

\begin{abstract}
 Recently, Kothari et al.\ gave an algorithm for testing the surface
 area of an arbitrary set $A \subset [0, 1]^n$. Specifically, they
 gave a randomized algorithm such that if $A$'s surface area
 is less than $S$ then the
 algorithm will accept with high probability, and if the algorithm
 accepts with high probability then there is some perturbation of
 $A$ with surface area at most $\kappa_n S$. Here, $\kappa_n$
 is a dimension-dependent constant which is strictly larger than 1
 if $n \ge 2$, and grows to $4/\pi$ as $n \to \infty$.
 
 We give an improved analysis of Kothari et al.'s algorithm.
 In doing so, we replace the constant $\kappa_n$ with
 $1 + \eta$ for $\eta > 0$ arbitrary. We also extend the
 algorithm to more general measures on Riemannian manifolds.
\end{abstract}

\category{Mathematics of Computing}{Probability and Statistics}{Probabilistic Algorithms}\\
\category{Theory of Computation}{Randomness, Geometry, and Discrete Structures}
{Computational Geometry}
\terms{surface area, noise sensitivity, property testing}

\section{Introduction}

Consider the problem of estimating the surface area of
a set $A \subset \R^n$. This fundamental geometric problem has been widely
studied, thanks in part to its applications in computer graphics
(see, e.g.,~\cite{LYZZP:10}) and medical imaging (see, e.g.,~\cite{BaGuCr:86,CuFrRo:07}).
Due to its diverse applications,
surface area estimation has been studied under several
different models. In this work, we consider a query access model,
in which one can choose points in $\R^n$ and ask whether then belong to $A$;
this is perhaps the most natural model for applications to learning.
In contrast, 
in stereography~\cite{BaGuCr:86} one considers the problem
of estimating $A$'s surface area given complete access to many random projections
(or slices) of $A$; this is a natural model for certain problems in 3D imaging.

\paragraph{Testing surface area}
Kearns and Ron~\cite{KearnsRon:98} introduced the study of surface area
in the property testing framework. In this framework, which goes back
to~\cite{BlLuRu:93},
we fix some predicate $P$
and ask for an algorithm that can distinguish (with high probability)
problem instances
satisfying $P$ from instances that are far from satisfying $P$.
For the specific case of testing the surface area of a set $A \subset [0, 1]^n$,
we denote the Lebesgue measure by $\lambda_n$, its corresponding surface area
measure by $\lambda_n^+$, and
we ask for an algorithm that accepts if $\lambda_n^+(A) \le S$ and rejects if
$\lambda_n^+(B) \ge (1 + \eta) S$ for all $B$ satisfying $\lambda(B \symdiff A) \le \epsilon$.
This problem was first considered by Kearns and Ron~\cite{KearnsRon:98} in the case $n=1$.
They showed that it can be solved for any $\epsilon > 0$ with $O(1/\epsilon)$ queries,
provided that $\eta \ge 1/\epsilon$. The approximation error was improved by
Balcan et al.~\cite{BBBY:12}, who gave an algorithm (still for $n=1$) that requires $O(\epsilon^{-4})$
queries but works for $\eta = 0$.
Substantial progress was made by
Kothari et al.~\cite{KNOW:13}, who gave an algorithm that works in arbitary
dimensions, but requires $\eta$ to be bounded away from zero as soon as $n \ge 2$.

Our main result is an improvement of Kothari et al.'s analysis, showing that
for any $\epsilon, \eta > 0$, an essentially identical
algorithm requires $O(\eta^{-3} \epsilon^{-1})$ samples, and guarantees that if the algorithm
rejects with substantial probability then
$\lambda_n^+(B) \ge (1 + \eta) S$ for all $B$ satisfying $\lambda(B \symdiff A) \le \epsilon$.
Besides the improvement in Kothari et al.'s soundness condition, our other main
contribution is a unified argument that applies to more general spaces than
$[0, 1]^n$ with the Lebesgue measure. 

\paragraph{Other work on estimating surface area}
The property testing formulation of surface area measurement asks for fairly
weak guarantees compared to
the more general problem of estimating the surface area of $A$.
One advantage of these weaker requirements is that the problem can be solved in
much greater generality, and with fewer queries. For example, the state of the
art in surface area estimation for convex sets is an algorithm of
Belkin et al.~\cite{BeNaNi:06} which requires $\tilde O(n^4 \epsilon^{-2})$ queries. 
One obvious difficulty in moving beyond convex sets is that a general set $A$ may have
large surface area, but only because of some very small part of it in which the boundary
wiggles crazily. Cuevas et al.~\cite{CuFrRo:07} dealt with this issue by imposing some
regularity on the boundary of $A$, but their algorithm required $O(\epsilon^{-2n})$
samples.

The property testing framework deals with the ``crazy boundary'' issue in
quite a simple way. Essentially, we relax the soundness condition so that if
$A$ can be modified on a set of measure $\epsilon$ in order to reduce its
surface area below $(1 + \eta) S$, then we allow the algorithm to accept.

\paragraph{Surface area and learning}
One application for surface area testing -- which was noted already
by Kothari et al. -- is in the ``testing before learning'' framework.
Indeed, Klivans et al.~\cite{KlOdSe:08} gave an algorithm for agnostically
learning sets in $\R^n$ under the Gaussian measure; their algorithm
runs in time $n^{O(S^2)}$, where $S$ is the Gaussian surface area
of the set that they are trying to learn. In particular, if we could
test the Gaussian surface area of a set before trying to learn it,
we would have advance notice of the learning task's complexity.
(Although we have not yet mentioned
Gaussian surface area, our results hold just as well in that case as
they do for Lebesgue surface area.)

\paragraph{Surface area and noise sensitivity}

Consider the torus $\T^n = (\R / \Z)^n$ with the Lebesgue measure $\lambda_n$
(from now on, for technical convenience, we will consider sets $A \subset \T^n$ instead of sets
$A \subset [0, 1]^n$).
Let $X$ be a uniformly distributed point in $\T^n$ and set $Y = X + \sqrt{2t} Z$,
where $Z \sim \normal(0, I_n)$ is a standard Gaussian vector.
For a set $A \subset \T^n$, we define the \emph{noise sensitivity} of
$A$ at scale $t$ by
\[
 \NS_t(A) = \Pr(X \in A, Y \not \in A) + \Pr(Y \in A, X \not \in A).
\]

Crofton, inspired by the Comte de Buffon's famous needle problem,
was the first to make a connection between surface area and noise sensitivity. 
His classical formula (see, e.g.,~\cite{Santalo:53}) implies that if $A \subset \T^n$
is a set with $\calC^1$ boundary then the surface area of $A$ is equal to
$\frac{2\sqrt t}{\sqrt \pi}$ times
the expected number of times that the line segment $[X, Y]$ crosses $\partial A$.
Since this number of crossings is always at least 1 on the event $\{1_A(X) \ne 1_A(Y)\}$,
we have the inequality
\begin{equation}\label{eq:ns-sa}
 \NS_t(A) \le \frac{2\sqrt t}{\sqrt\pi} \lambda_n^+(A),
\end{equation}
where $\lambda_n^+$ denotes the surface area.

The inequality~\eqref{eq:ns-sa} cannot be reversed in general.
To construct a counter-example, note that $A$ may be modified on a set of
arbitrarily small measure (which will affect the left hand side
of~\eqref{eq:ns-sa} by an arbitrarily small amount) while increasing its
surface area by an arbitrarily large amount. The main geometric result of this work
is that when $t$ is small, these counter-examples to a reversal
of~\eqref{eq:ns-sa} are essentially
the only ones possible.

\begin{theorem}\label{thm:lebesgue}
 For any $A \subset \T^n$ with $\calC^1$ boundary, and for every $\eta, t > 0$,
 there is a set $B \subset \T^n$ with $\lambda_n(A \symdiff B) \le \NS_t(A) / \eta$
 and
 \[
 \frac{2\sqrt t}{\sqrt\pi}
 \lambda_n^+(B) \le (1 + o(\eta)) \NS_t(A).
 \]
\end{theorem}

\begin{remark}
 One might wonder whether the $(1 + o(\eta))$ term on the right hand
 side of Theorem~\ref{thm:lebesgue} may be removed. Although we cannot rule
 out some improvement of Theorem~\ref{thm:lebesgue}, one should keep in
 mind the example of the ``dashed-line'' set
 $A_t \subset \T^1$ consisting of $\lfloor 5/\sqrt t \rfloor$ intervals of length
 $\sqrt t/10$, separated by intervals of length $\sqrt t/10$.
 Then $\NS_t(A_t)$ is a constant factor smaller
 than $2 \sqrt t / \lambda_1^+(A_t) / \sqrt{\pi}$
 as $t \to 0$. Moreover, reducing $\lambda_1^+(A_t)$ by
 a constant factor would require changing $A_t$ on a set of constant measure.
 In other words, it is not possible to simultaneously have
 $\lambda_1^+(A_t \symdiff B) = o(1)$ and
 \[
 \frac{2\sqrt t}{\sqrt\pi}
 \lambda_n^+(B) \le \NS_t(A_t)
 \]
 as $t \to 0$. On the other hand, if we let $\eta = \omega(1)$ in
 Theorem~\ref{thm:lebesgue} increase sufficiently slowly, we see that for
 any sequence $A_t$, we can find $B_t$ with
 $\lambda_1^+(A_t \symdiff B_t) = o(1)$ and
 \[
 \frac{2\sqrt t}{\sqrt\pi}
 \lambda_n^+(B_t) \le (1 + o(1))\NS_t(A_t).
 \]

\end{remark}

Theorems of this sort were introduced by Kearns and Ron~\cite{KearnsRon:98},
and by Balcan et al.~\cite{BBBY:12} in
dimension 1, and extended to $\T^n$ by Kothari et al.~\cite{KNOW:13}.
However, Kothari et al.\ gave a factor of $\kappa_n + \eta$ instead of $1 + \eta$
on the right hand side, where $\kappa_n$ is an explicit constant that grows
from $1$ to $4/\pi$ as $n$ goes from 1 to $\infty$. In fact, our analysis
will be closely based on that of~\cite{KNOW:13}; our main contribution
is an improved use of certain smoothness estimates,
leading to an improved constant.

With~\eqref{eq:ns-sa} and Theorem~\ref{thm:lebesgue} in hand, the algorithm
for testing surface area is quite simple. By sampling $m$ pairs
$(X, Y)$ according to the distribution above, one can estimate 
$\NS_t(A)$ to accuracy $O(m^{-1/2} \sqrt{\NS_t(A)})$. Consider, then, the algorithm
that says ``yes'' if and only if this estimate is smaller than
$2 \sqrt{t/\pi} (S + O(m^{-1/2} t^{-1/4} S^{1/2}))$.
The completeness of the algorithm then follows immediately from~\eqref{eq:ns-sa}
and Chebyshev's inequality.

The preceding algorithm is complete for any $m$ and $t$, but to show its soundness
we will need to specify $m$ and $t$. First, we will assume that $m^{-1/2} t^{-1/4} S^{1/2} \le \eta$.
An application of Chebyshev's inequality then shows that for $A$
to be accepted with high probability,
$A$ must satisfy $\NS_t(A) \le 2 \sqrt{t/\pi} (1 + O(\eta)) S$.
Applying Theorem~\ref{thm:lebesgue}, we see that there is some set $B$ with
\[
\lambda_n(A \symdiff B) \le O(S t^{1/2} \eta^{-1})
\]
and
\[
\lambda_n^+(B) \le \frac{\sqrt\pi}{2\sqrt{t}}(1 + o(\eta)) \NS_t(A) \le (1 + O(\eta)) S.
\]
Hence, the algorithm is sound provided that $S t^{1/2} \eta^{-1} \le \epsilon$.
Putting this condition together with the previous one, we see that
$t = (\epsilon \eta / S)^2$ and $m = O(\eta^{-3} \epsilon^{-1} S^2)$ suffices.
We summarize this discussion with the following algorithm.

\begin{algorithm}
 \SetKwInOut{Input}{Input}
 \SetKwInOut{Output}{Output}
 \setlength{\algomargin}{2em}
 \Input{query access to $A \subset \T^n$, \\
 parameters $S, \eta, \epsilon > 0$}
 \Output{with probability at least $2/3$, \\
 ``yes'' if $\lambda_n^+(A) \le S$, and \\
 ``no'' if $\lambda_n^+(B) \ge (1 + \eta) S$ for all $B$ with
 $\lambda_n(A \symdiff B) \le \epsilon$}
 
 \BlankLine
 $t \leftarrow \left(\frac{\epsilon \eta}{S}\right)^2$ \;
 $m \leftarrow 7 \eta^{-3} \epsilon^{-1} S^2$ \;
 Sample $X_1, \dots, X_m$ i.i.d. uniformly in $\T^n$\;
 Sample $Z_1, \dots, Z_m$ i.i.d. from $\normal(0, I_n)$\;
 For each $i = 1, \dots, m$, set $Y_i = X_i + \sqrt{2t} Z_i$\;
 Output ``yes'' iff $\frac{1}{m}\#\{i: 1_A(X_i) \ne 1_A(Y_i)\} \le \frac{2\sqrt t}{\sqrt \pi} (S + m^{-1/2}t^{-1/4} S^{1/2})$\;
 \caption{testing surface area}
 \label{alg}
\end{algorithm}

\paragraph{Surface area testing in Gaussian and other spaces}
Our analysis is not specific to the Lebesgue measure on the torus.
For example, Theorem~\ref{thm:lebesgue} also holds if
$\lambda_n$ is replaced by the Gaussian measure and
$\NS_t(A)$ is replaced by
$\Pr(1_A(Z) \ne 1_A(\rho Z + \sqrt{1-\rho^2} Z'))$, where $Z$ and $Z'$ are
independent Gaussian vectors on $\R^n$.
This Gaussian case was also considered
in~\cite{KNOW:13}, who obtained the same result but with an extraneous factor
of $4/\pi$ on the right hand side.
Since there is an analogue of~\eqref{eq:ns-sa} also in the Gaussian case
(due to Ledoux~\cite{Ledoux:94}), one also obtains an algorithm for
testing Gaussian surface area.

More generally, one could ask for a version of Theorem~\ref{thm:lebesgue}
on any weighted manifold.
We propose a generalization of Theorem~\ref{thm:lebesgue}
in which the noise sensitivity is measured with respect to a Markov
diffusion
semigroup and the surface area is measured with respect to that semigroup's
stationary measure. The class of stationary measures allowed by
this extension includes log-concave measures on $\R^n$ and Riemannian
volume elements on compact manifolds.
Since Ledoux's argument~\cite{Ledoux:94} may be extended in this
generality, our algorithm again extends.

\section{Markov semigroups and curvature}

As stated in the introduction, we will carry out the proof of
Theorem~\ref{thm:lebesgue} in the setting of Markov diffusion semigroups. 
An introduction to this topic may be found in the Ledoux's
monograph~\cite{Ledoux:00}. To follow our proof, however, it is not
necessary to know the general theory; we will be concrete about
the Gaussian and Lebesgue cases, and it is enough to keep one of these
in mind.

Let $(M, g)$ be a smooth, Riemannian $n$-manifold and consider the
differential operator $L$ that is locally defined by
\begin{equation}\label{eq:L}
 (L f)(x) = \sum_{i,j = 1}^n g^{ij}(x) \pdiffII{f}{x_i}{x_j}
 + \sum_{i=1}^n b^i(x) \pdiff{f}{x_i}
\end{equation}
where $b^i$ are smooth functions and
$(g^{ij} (x))_{i,j=1}^n$ is the inverse tensor of $g$ in local coordinates.
Such an operator induces a semigroup $(P_t)_{t \ge 0}$ of operators which
satisfies $\diff{}{t} P_t f = L f$.
There are certain technical issues, which we will gloss over here,
regarding the domains of these operators. We will assume that
the domain of $L$ contains an algebra $\calA$ satisfying $P_t \calA \subset \calA$.
We will assume moreover that $P_t$ has an invariant probability distribution
$\mu$ which is absolutely continuous with respect to the Riemannian
volume element on $M$; we will also assume that $\calA$ is dense in $L_p(\mu)$ for
every $\mu$.
In any case, these assumptions are satisfied in many interesting examples,
such as when $P_t$ is the heat semigroup on a compact Riemannian
manifold,
or when $P_t$ is the Markov semigroup associated with a log-concave
measure $\mu$ on $\R^n$. See~\cite{Ledoux:00} for more details.

Given a Markov semigroup $P_t$, we define the noise sensitivity
by
\begin{equation}\label{eq:ns-def}
\NS_t(A) = \int_M |P_t 1_A(x) - 1_A(x)| \, d\mu(x).
\end{equation}
The probabilistic interpretation of this quantity is given by the Markov
process associated with $P_t$. This is a Markov process $(X_t)_{t \in \R}$
with the property that for any $f \in L_1(\mu)$,
$\E (f(X_t) \mid X_0) = (P_t f)(X_0)$. Given such a process,
the noise sensitivity may be alternatively written as
$\NS_t(A) = \Pr(1_A(X_0) \ne 1_A(X_t))$.

The other notion we need is that of surface area. Recalling that
$\mu$ was assumed to have a density with respect to the Riemannian volume,
we define
\[
 \mu^+(A) = \int_{\partial A} \mu(x) d\calH_{n-1}(x),
\]
where $\calH_{n-1}$ is the $(n-1)$-dimensional Hausdorff measure.

Let us revisit $(\T^n, \lambda_n)$ and $(\R^n, \gamma_n)$ in our
more abstract setting. 
In the case of $\T^n$, we set $L$ to be
$\sum_{i=1}^n \pdiffII{}{x_i}{x_i}$. Then $P_t$ is given by
\[
 (P_t f)(x) = \int_{\R^n} f(x + \sqrt {2t} y)\, d\gamma_n(y).
\]
The associated Markov process $X_t$ is simply Brownian motion, and so
we see that the noise sensitivity defined in~\eqref{eq:ns-def}
coincides with the definition that we gave in the introduction.

In the Gaussian case, we define $L$ by
\[
(Lf)(x) = \sum_{i=1}^n \left(\pdiffII{f}{x_i}{x_i} - x_i \pdiff{f}{x_i}\right).
\]
Then $P_t$ is given by
\[
 (P_t f)(x) = \int_{\R^n} f(e^{-t} x + \sqrt{1-e^{-2t}} y)\, d\gamma_n(y).
\]
The associated Markov process $X_t$ is the Ornstein-Uhlenbeck process,
which is the Gaussian process for which $\E X_s^T X_t = e^{-|s-t|} I_n$.

In order to state our generalization of Theorem~\ref{thm:lebesgue}, we
need a geometric condition on the semigroup $P_t$.
Following Bakry and Emery~\cite{BakryEmery:85} (see also~\cite{Ledoux:00}),
we say that the semigroup $(P_t)_{t \ge 0}$ has curvature $R$ if the inequality
$|\grad P_t f| \le e^{-R t} P_t |\grad f|$ holds pointwise
for every $f \in \calA$. One can check easily from the definitions
that in the heat semigroup on $\T^n$ has curvature $0$, while the
Ornstein-Uhlenbeck semigroup on $\R^n$ has curvature $1$.

In this abstract setting, our main theorem is the following.
Observe that by specializing to the torus $\T^n$ with the Lebesgue
measure $\lambda_n$ and curvature $R = 0$, we recover
Theorem~\ref{thm:lebesgue}.
\begin{theorem}\label{thm:main}
Suppose that $P_t$ has curvature $R$.
For any $A \subset M$ with $\calC^1$ boundary and for every
$\eta, t > 0$, there is a set $B \subset M$ with $\mu(A \symdiff B)
\le \NS_t(A) / \eta$ and
\[
 \mu_n^+(B)
 \le \sqrt{\frac{\pi}{2}} \left(1 + \frac{\sqrt \pi \eta}{\sqrt{\log(1/\eta)}}(1 + o_\eta(1))\right) c_R(t) \NS_t(A),
\]
 where $c_R(t) = \left(\frac{e^{2Rt} - 1}{R}\right)^{-1/2}$ if $R \ne 0$
 and $c_0(t) = (2t)^{-1/2}$.
 Here, $o_\eta(1)$ denotes a quantity that tends to zero as $\eta \to 0$.
\end{theorem}

In order to prove Theorem~\ref{thm:main}, we will construct the set $B$
in a randomized way, using a construction that is due to
Kothari et al.~\cite{KNOW:13}.
Their construction is quite simple: we first
smooth the function $1_A$ using $P_t$ and then threshold $P_t 1_A$ to obtain
$1_B$. One difficulty
with this procedure is to find a suitable threshold value. Kothari et al.\ dealt
with this difficulty in a remarkably elegant way: they showed that after
thresholding at an appropriately chosen random value, the expected surface area
of the resulting set is small. In particular, there is some threshold value
that suffices.

The analysis of the random thresholding procedure uses two main tools:
the first is the coarea formula (see, e.g.~\cite{Federer:69}), which will
allow us
to express the expected surface area of our thresholded set in terms of the
gradient of $P_t 1_A$.

\begin{theorem}[Coarea formula]\label{thm:coarea}
For any $\calC^1$ function $f: M \to [0, 1]$,
any $\mu \in L_1(M)$, and any
$\psi \in L_\infty([0, 1])$,
\begin{multline*}
 \int_0^1 \psi(s) \int_{\{x \in M: f(x) = s\}}\mu(x)\, d\calH_{n-1}(x) \, ds \\
 = \int_M \psi(f(x)) |\grad f(x)|\mu(x)\, dx.
\end{multline*}
\end{theorem}

Our second tool is a pointwise bound on $|\grad P_t f|$ for any
$f: M \to [0, 1]$. This will allow us, after applying the coarea
formula, to obtain a sharp bound on the integral involving $|\grad P_t 1_A|$.

\begin{theorem}[\!\!\cite{BakryLedoux:96}]\label{thm:smoothness}
 If $P_t$ has curvature $R$ then for any $f: M \to [0, 1]$ and
 any $t > 0$,
 \[
  |\grad P_t f| \le c_R(t) I(P_t f),
 \]
 where $c_R(t) = \left(\frac{e^{2Rt} - 1}{R}\right)^{-1/2}$ if $R \ne 0$
 and $c_0(t) = (2t)^{-1/2}$.
\end{theorem}

 For $g: M \to [0, 1]$, let $g^{\ge s}$ denote the set $\{x \in M: g(x) \ge s\}$.
 If $g$ is continuous then
 $\partial g^{\ge s} = \{x \in M: g(x) = s\}$. Hence the surface area of
 $g^{\ge s}$ is simply
 \[
 \mu^+(g^{\ge s}) = \int_{\{x \in M: g(x) = s\}} \mu(x) \, d\calH_{n-1}(x),
 \]
 and so the coarea formula (Theorem~\ref{thm:coarea}) implies that
 \[
  \int_0^1 \psi(s) \mu^+(g^{\ge s})\, ds
  = \int_M \psi(g(x)) |\grad g(x)| \mu(x)\, dx = \E \psi(g) |\grad g|.
 \]
 (From here on, we will often write $\E$ for the integral with
 respect to $\mu$ which, recall, is a probability measure.)
 On the other hand, $\int_0^1 \psi(s) \mu^+(g^{\ge s})\, ds
 \ge \min_{s \in [0, 1]} \mu^+(g^{\ge s}) \int_0^1 \psi(s)\, ds$. In particular,
 if we can show that $\E \psi(g) |\grad g|$ is small then it will follow
 that $\mu^+(g^{\ge s})$ is small for some $s$.
 
 Unsurprisingly, the quantity $\E \psi(g) |\grad g|$ is quite sensitive to
 the choice of $\psi$. In order to get the optimal constant in
 Theorem~\ref{thm:main}, we need to choose a particular function $\psi$.
 Namely, we define
 \[
  \psi(s) = \frac{\frac 12 - \left|s - \frac 12\right|}{I(s)}.
 \]

 \begin{lemma}\label{lem:coarea}
  For any measurable $A \subset M$ and any $t > 0$,
  \[
   \E \psi(P_t 1_A) |\grad P_t 1_A| \le c_R(t) \NS_t(A).
  \]
 \end{lemma}

 \begin{proof}
 By Theorem~\ref{thm:smoothness},
 \begin{align*}
  \E \psi(P_t 1_A) |\grad P_t 1_A|
  &\le c_R(t) \E \psi(P_t 1_A) I(P_t 1_A) \\
  &=
  c_R(t)\E \left(\frac 12 - \left|P_t 1_A - \frac 12\right|\right).
 \end{align*}
 Now, $\frac 12 - |x - \frac 12| = \min\{x, 1-x\}$ and so
 \begin{align*}
  \E \left(\frac 12 - \left|P_t 1_A - \frac 12\right|\right)
  &= \E \min\{P_t 1_A, 1-P_t 1_A\} \\
  &\le \E |P_t 1_A - 1_A| \\
  &= \NS_t(A).
 \end{align*}
 \end{proof}
 
 Going back to the discussion before Lemma~\ref{lem:coarea}, we have
 shown that
 \begin{align*}
  \min_{s \in [0, 1]}
  \mu^+((P_t 1_A)^{\ge s}) \int_0^1 \psi(s)\, ds
  &\le \int_0^1 \mu^+((P_t 1_A)^{\ge s}) \psi(s)\, ds \\
  &\le c_R(t) \NS_t(A).
 \end{align*}
 Since we are concerned in this work with optimal constants, let us
 compute
 $\int_0^1 \psi(s)\, ds$:
 \begin{lemma}\label{lem:integral}
  $\displaystyle \int_0^1 \psi(s)\, ds = \sqrt{\frac 2\pi}$.
 \end{lemma}
 
 \begin{proof}
  We use the substitution $s = \Phi(y)$. Then $ds = \phi(y)\, dy$ and
  $I(s) = \phi(y)$. Hence,
  \[
   \int_0^1 \psi(s)\, ds
   = \int_{-\infty}^\infty \frac 12 - \left|\frac 12 - \Phi(y)\right|\,dy
   = 2 \int_0^\infty 1 - \Phi(y)\, dy,
  \]
  where the last equality follows because $\Phi(-t) = 1-\Phi(t)$ and
  $\Phi(t) \ge \frac 12$ for $t \ge 0$. Recalling the definition
  of $\Phi$, if we set $Z$ to be a standard Gaussian variable then
  \[
   2 \int_0^\infty 1 - \Phi(y)\, dy = 2 \E \max\{0, Z\} = \E |Z| =
   \sqrt{\frac 2\pi}.
  \]
 \end{proof}
 
 Combining Lemmas~\ref{lem:integral} and~\ref{lem:coarea}, we have shown
 the existence of some $s \in [0, 1]$ such that
 $\mu^+((P_t 1_A)^{\ge s}) \le \sqrt{\pi/2} c_R(t) \NS_t(A)$.
 This is not quite enough to prove Theorem~\ref{thm:main} because we
 need to produce a set $B$ such that $\mu(B \symdiff A)$ is small.
 In general, $(P_t 1_A)^{\ge s}$ may not be close to $A$;
 however, if $s \in [\eta, 1-\eta]$ then they are close:
 \begin{lemma}\label{lem:closeness}
  For any $t > 0$, if $s \in [\eta, 1-\eta]$ then
  \[
   \mu((P_t 1_A)^{\ge s} \symdiff A) \le \frac 1\eta \NS_t(A).
  \]
 \end{lemma}
 
 \begin{proof}
  Note that if the indicator of $(P_t 1_A)^{\ge s}$ is not equal to
  $1_A$ then either
  $1_A = 0$ and $P_t 1_A \ge s$ or $1_A = 1$ and $P_t 1_A < s$. If
  $s \in [\eta, 1-\eta]$ then either case implies that
  $|P_t 1_A - 1_A| \ge \eta$. Hence,
  \[
   \mu((P_t 1_A)^{\ge s} \symdiff A) \le \frac 1\eta \E |P_t 1_A - 1_A|
   = \frac 1\eta \NS_t(A).
  \]
 \end{proof}

To complete the proof of Theorem~\ref{thm:main}, we need to
invoke Lemmas~\ref{lem:coarea} and~\ref{lem:integral} in a slightly different
way from before. Indeed, with Lemma~\ref{lem:closeness} in mind we want to
show that there is some $s$ for which $\mu^+((P_t 1_A)^+)$ is small
\emph{and} such that $s$ is not too close to zero or one. For this, we note that
\begin{align*}
 \int_\eta^{1-\eta} \psi(s) \, ds
 \min_{s \in [\eta, 1-\eta]} \mu^+(g^{\ge s})
 &\le 
 \int_\eta^{1-\eta} \psi(s) \mu^+(g^{\ge s})\, ds \\
 &\le
 \int_0^1 \psi(s) \mu^+(g^{\ge s})\, ds.
\end{align*}
With $g = P_t 1_A$, we see from Lemma~\ref{lem:coarea} that
\begin{equation}\label{eq:min-away-from-bdy}
 \min_{s \in [\eta, 1-\eta]} \mu^+((P_t 1_A)^{\ge s})
 \le \frac{c_R(t) \NS_t(A)}{\int_{\eta}^{1-\eta} \psi(s)\, ds}.
\end{equation}
To compute the denominator, one checks (see, e.g.,~\cite{BakryLedoux:96})
the limit $I(x) \sim x \sqrt{2 \log(1/x)}$ as $x \to 0$ and so
$\psi(x) \sim (2 \log(1/x))^{-1/2}$ as $x \to 0$. Hence,
$\int_0^\eta \psi(s)\, ds \sim \eta (2 \log(1/\eta))^{-1/2}$ as
$\eta \to 0$ and since $\psi(x)$ is symmetric around $x=1/2$,
\[
 \int_\eta^{1-\eta} \psi(s)\, ds
 = \int_0^1 \psi(s)\, ds - \frac{\sqrt 2 \eta}{\sqrt{\log(1/\eta)}}(1 + o(1))
\]
as $\eta \to 0$. Applying this to~\eqref{eq:min-away-from-bdy}
(along with the formula from Lemma~\ref{lem:integral}), there must
exist some $s \in [\eta, 1-\eta]$ with
\begin{align*}
 \mu^+((P_t 1_A)^{\ge s})
 &\le \frac{c_R(t) \NS_t(A)}{\int_{\eta}^{1-\eta} \psi(s)\, ds} \\
 &\le \sqrt{\frac{\pi}{2}} \left(1 + \frac{\sqrt \pi \eta}{\sqrt{\log(1/\eta)}}(1 + o(1))\right) c_R(t) \NS_t(A).
\end{align*}
Taking $B = (P_t 1_A)^{\ge s}$ for such an $s$, we see from
Lemma~\ref{lem:closeness} that this $B$ satisfies the claim of
Theorem~\ref{thm:main}, thereby completing the proof of that theorem.

\section{Acknowledgements}
We would like to thank E. Mossel for helpful discussions, and M. Ledoux for pointing out
that the argument of~\cite{Ledoux:94} may be extended.
Much of the work on this paper was done while the author was visiting the Simons Institute
for the Theory of Computing, which he thanks for its hospitality.

The author is partially supported by 
DTRA grant HDTRA1-13-1-0024.

\bibliographystyle{abbrv}
\bibliography{all}
\end{document}